\documentclass[11pt,reqno]{amsart}

\usepackage[letterpaper,margin=1in]{geometry}
\usepackage{amssymb,amsmath,amsthm,enumerate,tikz}

\usepackage[colorlinks]{hyperref}

\usepgflibrary[arrows]
\usetikzlibrary{arrows,matrix}

\newtheorem{prop}{Proposition}[section]
\newtheorem{lem}[prop]{Lemma}
\newtheorem{thrm}[prop]{Theorem}

\theoremstyle{definition}
\newtheorem{defn}[prop]{Definition}
\newtheorem{ex}[prop]{Example}

\theoremstyle{remark}
\newtheorem{rmk}[prop]{Remark}

\newcommand{\F}{\mathcal F}

\newcommand{\wt}[1]{\widetilde #1}
\newcommand{\red}{\twoheadrightarrow}
\newcommand{\cored}{\rightarrowtail}
\renewcommand{\bar}[1]{\overline{#1}}

\DeclareMathOperator{\dom}{dom}

\title{Cores of Symplectic Double Groupoids Via Reduction}
\author{Santiago Ca\~nez}
\address{Department of Mathematics, Northwestern University, 2033 Sheridan Rd, Evanston IL 60208, USA}
\email{scanez@northwestern.edu}
\date{}

\begin{document}

\begin{abstract}
We use symplectic reduction to give a new construction of the core $C$ of a symplectic double groupoid $D$ as the common leaf space of characteristic foliations associated to various coisotropic submanifolds of $D$. In the case of the cotangent double groupoid of a Lie groupoid $G$, the canonical relations arising from this process turn out to be cotangent lifts of structure maps associated to $G$. We also show that under this reduction procedure the double groupoid structure on $D$ descends to a groupoid structure on the leaf space above, recovering the core groupoid structure on $C$ of Brown and Mackenzie.
\end{abstract}

\maketitle\thispagestyle{empty}

\section{Introduction}
A symplectic double groupoid is a symplectic manifold equipped with two compatible symplectic groupoid structures. Such objects arise naturally in the study of Poisson groupoids and Lie bialgebroids, and have proved to be useful in the quantization of Poisson-Lie groups (see for instance \cite{SZ}). Their theory, as well as that for double Lie groupoids in general, is well-developed~\cite{BM},\cite{M},\cite{LW}. The core of a double Lie groupoid $D$ is a submanifold of $D$ which encodes parts of the double groupoid structure, and pops up naturally in various constructions, such as in describing the symplectic double groupoid structure inherited by the cotangent bundle of a double Lie groupoid (see Theorem~\ref{thrm:ctdbl}). It is well-known that the core of a symplectic double groupoid itself naturally inherits the structure of a symplectic groupoid~\cite{M}. 

In this paper we describe a procedure for producing the core of a symplectic double grouopid $D$ together with its symplectic groupoid structure via symplectic reduction. Here by symplectic reduction we mean the process of taking the quotient of a coisotropic submanifold of a symplectic manifold by its characteristic foliation. Such a reduction gives rise to a canonical relation---i.e. a lagrangian submanifold of a product---between the original symplectic manifold and the reduced one. In our case, this reduction procedure is applied to various natural coisotropic submanifolds of $D$ encoded by the double groupoid structure, and the resulting canonical relations allow us to transport structures on $D$ to its core.

These results should be viewed as part of the program where the symplectic category---i.e. the ``category'' of symplectic manifolds where morphisms are given by canonical relations---is used to provide natural descriptions of various constructions in symplectic geometry. Indeed, the motivation for this work was the observation that for a Lie groupoid $G \rightrightarrows M$, the core $T^*M$ of the standard double groupoid structure of $T^*G$ (see Example~\ref{ex:main}) can be obtained via symplectic reduction and the resulting reduction relations endow $T^*G \rightrightarrows T^*M$ with a natural groupoid-like structure in the symplectic category. Such structures will be studied in detail in a future paper~\cite{C} where it will be shown that they provide an alternate characterization of symplectic double groupoids.

This paper is structured as follows. In Section~\ref{sec:can-rel} we recall some basic facts about canonical relations and Zakrzewki's characterization of symplectic groupoids in terms of such objects. Section~\ref{sec:dbl-grpds} gives background material on symplectic double groupoids and their cores and sets up notations we will use. Finally, Section~\ref{sec:core-red}, which constitutes the bulk of the paper, contains the construction of cores via symplectic reduction and reproduces their groupoid structure using canonical relations. The main technical result which makes this possible is Lemma~\ref{lem:fol}, which emphasizes the role of the entire double groupoid structure.

\subsubsection*{Acknowledgements}
This work is based on the author's Ph. D. thesis~\cite{C1}. I would like to thank my advisor, Alan Weinstein, for his many years of support and encouragement. Thanks also to Rajan Mehta for numerous comments and suggestions regarding the version of these results which appeared in the thesis cited above.

\section{Canonical Relations}\label{sec:can-rel}
\begin{defn}
A \emph{canonical relation} $R$ from a symplectic manifold $M$ to a symplectic manifold $N$ is a closed lagrangian submanifold of $\bar M \times N$, where $\bar M$ denotes $M$ with the opposite symplectic structure. We will use the notation $R: M \to N$ to mean that $R$ is a canonical relation from $M$ to $N$, and the notation $R: m \mapsto n$ to mean that $(m,n) \in R$. The \emph{transpose} of $R$ is the canonical relation $R^t: N \to M$ defined by the condition that $(n,m) \in R^t$ if $(m,n) \in R$.
\end{defn}

\begin{ex}
The graph of a symplectomorphism $f: P \to Q$ is a canonical relation $P \to Q$, which by abuse of notation we will also denote by $f$. In particular, given any symplectic manifold $P$, the graph of the identity map will be the canonical relation $id: P \to P$ given by the diagonal in $\overline P \times P$. More generally, the graph of a symplectic \'etale map is a canonical relation.
\end{ex}


\begin{ex}
For any symplectic manifold $S$, a canonical relation $pt \to S$ or $S \to pt$ is nothing but a closed lagrangian submanifold of $S$.
\end{ex}

The \emph{domain} of a canonical relation $R: M \to N$ is the subset
\[
\dom L := \{m \in M\ |\ \text{there exists $n \in N$ such that } (m,n) \in L\}
\]
of $M$. We define the composition of canonical relations using the usual composition of relations: given canonical relations $R: M \to N$ and $R': N \to Q$, the composition $R' \circ R: M \to Q$ is
\[
R' \circ R := \{ (m,q) \in M \times Q\ |\ \text{there exists } n \in N \text{ such that } (m,n) \in R \text{ and } (n,q) \in R'\}.
\]
This is the same as taking the intersection of $R \times R'$ and $M \times \Delta_N \times Q$ in $M \times N \times N \times Q$, where $\Delta_N$ denotes the diagonal in $N \times N$, and projecting to $M \times Q$. However, we run into the problem that the above composition need no longer be a smooth closed submanifold of $M \times Q$, either because the intersection of $R \times R'$ and $M \times \Delta_N \times Q$ is not smooth or because the projection to $M \times Q$ is ill-behaved, or both. To fix this, we introduce the following notion:

\begin{defn}
A pair $(R,R')$ of canonical relations $R: M \to N$ and $R': N \to Q$ is \emph{strongly transversal} if the submanifolds $R \times R'$ and $M \times \Delta_N \times Q$ of $M \times N \times N \times Q$ intersect transversally and the projection of
\[
(R \times R') \cap (M \times \Delta_N \times Q)
\]
to $M \times Q$ is a proper embedding.
\end{defn}

As a consequence, for a strongly transversal pair $(R,R')$ the composition $R' \circ R$ is indeed a submanifold of $M \times Q$, and the following is well known:

\begin{prop}
If $L: X \to Y$ and $L': Y \to Z$ are canonical relations with $(L,L')$ strongly transversal, then $L' \circ L: X \to Z$ is a canonical relation.
\end{prop}

\begin{defn}
A canonical relation $R: M \to N$ is said to be:
\begin{itemize}
\item \emph{surjective} if for any $n \in N$ there exists $m \in M$ such that $(m,n) \in R$,
\item \emph{coinjective} if whenever $(m,n), (m,n') \in R$ we have $n=n'$.
\end{itemize}
\end{defn}

\begin{defn}
A canonical relation $R: M \to N$ is said to be a \emph{reduction} if it is surjective and coinjective, the projection of $R$ to $M$ is a proper embedding, and the projection of $R$ to $N$ is a submersion; it is a \emph{coreduction} if $R^t: N \to M$ is a reduction.
\end{defn}

\begin{rmk}
A pair $(R,R')$ of canonical relations is always strongly transversal if either $R$ is a reduction or $R'$ a coreduction.
\end{rmk}

The use of the term ``reduction'' is motivated by the following example.

\begin{ex}(Symplectic Reduction)
Let $(M,\omega)$ be a symplectic manifold and $C$ a coisotropic submanifold. The distribution on $C$ given by $\ker\omega \subset TC$, called the \emph{characteristic distribution} of $C$, is integrable and the induced foliation $C^\perp$ on $C$ is called the \emph{characteristic foliation} of $C$. If the leaf space $C/C^\perp$ is smooth and the projection $C \to C/C^\perp$ is a submersion, then $C/C^\perp$ naturally inherits a symplectic structure and the relation
\[
red: M \red C/C^\perp
\]
assigning to an element of $C$ the leaf which contains it is a canonical relation which is a reduction in the sense above. The construction of $C/C^\perp$ from $M$ and $C$ is called \emph{symplectic reduction}. Symplectic reduction via Hamiltonian actions of Lie groups is a special case.
\end{ex}

\begin{ex}\label{ex:std-reds}
We also note two more well-known examples of symplectic reduction. Suppose that $X$ is a manifold with $Y \subseteq X$ a submanifold. Then the restricted cotangent bundle $T^*X|_Y$ is a coisotropic submanifold of $T^*X$ whose reduction is symplectomorhpic to $T^*Y$. Thus we obtain a reduction $T^*X \twoheadrightarrow T^*Y$.

As a generalization, suppose now that $\F$ is a regular foliation on $Y \subseteq X$ with smooth, Hausdorff leaf space $Y/\F$. Then the conormal bundle $N^*\F$ is a coisotropic submanifold of $T^*X$ (this is in fact equivalent to the distribution $T\F$ being integrable) and its reduction is canonically symplectomorphic to $T^*(Y/\F)$, giving rise to a reduction relation $T^*X \twoheadrightarrow T^*(Y/\F)$. The previous example is the case where $\F$ is the zero-dimensional foliation given by the points of $Y$.
\end{ex}

To a smooth map $f: M \to N$ between manifolds $M$ and $N$ we can associate the canonical relation $T^*f: T^*M \to T^*N$ given by
\[
T^*f: (p,df_p^*\xi) \mapsto (f(p),\xi).
\]
This is nothing but the composition $T^*M \to \overline{T^*M} \to T^*N$ of the Schwartz transform of $T^*M$ followed by the canonical relation given by the conormal bundle to the graph of $f$ in $M \times N$. We call $T^*f$ the \emph{cotangent lift} of $f$.

\begin{ex}
When $\phi: M \to N$ is a diffeomorphism, $T^*\phi: T^*M \to T^*N$ is the graph of the usual lifted symplectomorphism.
\end{ex}

\begin{rmk}
Canonical relations can (almost) be viewed as the morphisms of a category known as the \emph{symplectic category}. In this setting, $T^*$ becomes a functor from the category of smooth manifolds to the symplectic category.
\end{rmk}

In the categorical framework mentioned above, the following is a natural notion to consider:

\begin{defn}
A \emph{symplectic monoid} is a triple $(S,m,e)$ consisting of a symplectic manifold $S$ together with canonical relations
\[
m: S \times S \to S \text{ and } e: pt \to S,
\]
called the \emph{product} and \emph{unit} respectively, so that
\begin{center}
\begin{tikzpicture}[>=angle 90]
	\node (UL) at (0,1) {$S \times S \times S$};
	\node (UR) at (3,1) {$S \times S$};
	\node (LL) at (0,-1) {$S \times S$};
	\node (LR) at (3,-1) {$P$};
	
	\tikzset{font=\scriptsize};
	\draw[->] (UL) to node [above] {$id \times m$} (UR);
	\draw[->] (UL) to node [left] {$m \times id$} (LL);
	\draw[->] (UR) to node [right] {$m$} (LR);
	\draw[->] (LL) to node [above] {$m$} (LR);
\end{tikzpicture}
\end{center}
and
\begin{center}
\begin{tikzpicture}[thick]
	\node (UL) at (0,1) {$S$};
	\node (UR) at (3,1) {$S \times S$};
	\node (URR) at (6,1) {$S$};
	\node (LR) at (3,-1) {$S$};

	\tikzset{font=\scriptsize};	
	\draw[->] (UL) to node [above] {$e \times id$} (UR);
	\draw[->] (UR) to node [right] {$m$} (LR);
	\draw[->] (URR) to node [above] {$id \times e$} (UR);
	\draw[->] (UL) to node [above] {$id$} (LR);
	\draw[->] (URR) to node [above] {$id$} (LR);
\end{tikzpicture}
\end{center}
commute. We require that all compositions involved be strongly transversal.
\end{defn}

\begin{ex}
For a symplectic groupoid $S \rightrightarrows P$, $S$ together with the groupoid multiplication thought of as a relation $S \times S \to S$ and the canonical relation $pt \to S$ given by the image of the unit embedding $P \to S$ is a symplectic monoid.
\end{ex}

Zakrzewski gave in \cite{SZ1}, \cite{SZ2} a complete characterization of symplectic groupoids in terms of such structures, or more specifically, symplectic monoids equipped with a $*$-structure:

\begin{defn}
A \emph{*-structure} on a symplectic monoid $S$ is an anti-symplectomorphism $s: S \to S$ (equivalently a symplectomorphism $s: \overline S \to S$) such that $s^2 = id$ and the diagram
\begin{center}
\begin{tikzpicture}[>=angle 90]
	\node (UL) at (0,1) {$\overline{S} \times \overline{S}$};
	\node (U) at (3,1) {$\overline{S} \times \overline{S}$};
	\node (UR) at (6,1) {$S \times S$};
	\node (LL) at (0,-1) {$\overline{S}$};
	\node (LR) at (6,-1) {$S,$};

	\tikzset{font=\scriptsize};
	\draw[->] (UL) to node [above] {$\sigma$} (U);
	\draw[->] (U) to node [above] {$s \times s$} (UR);
	\draw[->] (UL) to node [left] {$m$} (LL);
	\draw[->] (UR) to node [right] {$m$} (LR);
	\draw[->] (LL) to node [above] {$s$} (LR);
\end{tikzpicture}
\end{center}
where $\sigma$ is the symplectomorphism exchanging components, commutes. A symplectic monoid equipped with a $*$-structure will be called a \emph{symplectic $*$-monoid}.

A $*$-structure $s$ is said to be \emph{strongly positive} if the diagram
\begin{center}
\begin{tikzpicture}[>=angle 90]
	\node (UL) at (0,1) {$S \times \overline{S}$};
	\node (UR) at (3,1) {$S \times S$};
	\node (LL) at (0,-1) {$pt$};
	\node (LR) at (3,-1) {$S,$};

	\tikzset{font=\scriptsize};
	\draw[->] (UL) to node [above] {$id \times s$} (UR);
	\draw[->] (LL) to node [left] {$$} (UL);
	\draw[->] (UR) to node [right] {$m$} (LR);
	\draw[->] (LL) to node [above] {$e$} (LR);
\end{tikzpicture}
\end{center}
where $pt \to S \times \overline{S}$ is the morphism given by the diagonal of $S \times \overline{S}$, commutes.
\end{defn}

\begin{thrm}[Zakrzewski, \cite{SZ1}\cite{SZ2}]
Symplectic groupoids are in $1$-$1$ correspondence with strongly positive symplectic $*$-monoids.
\end{thrm}

It is this point of view of symplectic groupoids which we will use to describe the groupoid structure on the core of a double groupoid.

\section{Symplectic Double Groupoids}\label{sec:dbl-grpds}

\begin{defn}
A \emph{double Lie groupoid} is a diagram
\begin{equation}\label{defn:dbl-grpd}
\begin{tikzpicture}[>=angle 90,scale=.85,baseline=(current  bounding  box.center)]
	\node at (0,0) {$H$};
	\node at (2,0) {$M$};
	\node at (0,2) {$D$};
	\node at (2,2) {$V$};

	\tikzset{font=\scriptsize};
	\draw[->] (.4,.08) -- (1.6,.08);
	\draw[->] (.4,-.08) -- (1.6,-.08);
	
	\draw[->] (-.08,1.6) -- (-.08,.4);
	\draw[->] (.08,1.6) -- (.08,.4);
	
	\draw[->] (.4,2+.08) -- (1.6,2+.08);
	\draw[->] (.4,2-.08) -- (1.6,2-.08);
	
	\draw[->] (2-.08,1.6) -- (2-.08,.4);
	\draw[->] (2+.08,1.6) -- (2+.08,.4);
\end{tikzpicture}
\end{equation}
of Lie groupoids such that the structure maps of the top and bottom groupoids are homomorphisms between the left and right groupoids, and vice-versa. For technical reasons, we assume that the \emph{double source map}
\[
D \to H \times_M V
\]
is a surjective submersion. Here the fiber product is taken with respect to the source maps for $H \rightrightarrows M$ and $V \rightrightarrows M$. We will refer to the various groupoid structures in this diagram as the top, left, bottom, and right groupoids, and often refer to $D$ itself as the double Lie groupoid.
\end{defn}

\begin{rmk}
The double groupoid structure on $D$ should not depend on the manner in which we have chosen to draw the diagram above. In other words, we will think of
\begin{center}
\begin{tikzpicture}[>=angle 90,scale=.85]
	\node at (0,0) {$V$};
	\node at (2,0) {$M$};
	\node at (0,2) {$D$};
	\node at (2,2) {$H$};

	\tikzset{font=\scriptsize};
	\draw[->] (.4,.08) -- (1.6,.08);
	\draw[->] (.4,-.08) -- (1.6,-.08);
	
	\draw[->] (-.08,1.6) -- (-.08,.4);
	\draw[->] (.08,1.6) -- (.08,.4);
	
	\draw[->] (.4,2+.08) -- (1.6,2+.08);
	\draw[->] (.4,2-.08) -- (1.6,2-.08);
	
	\draw[->] (2-.08,1.6) -- (2-.08,.4);
	\draw[->] (2+.08,1.6) -- (2+.08,.4);
\end{tikzpicture}
\end{center}
as representing the same double groupoid as above, and will call this the \emph{transpose} of the previous structure.
\end{rmk}

\begin{ex}\label{ex:dmain}
For any Lie groupoid $G \rightrightarrows M$, there is a double Lie groupoid structure on
\begin{center}
\begin{tikzpicture}[>=angle 90]
	\node (LL) at (0,0) {$M$};
	\node (LR) at (2,0) {$M.$};
	\node (UL) at (0,2) {$G$};
	\node (UR) at (2,2) {$G$};

	\tikzset{font=\scriptsize};
	\draw[->] (.4,.07) -- (1.6,.07);
	\draw[->] (.4,-.07) -- (1.6,-.07);
	
	\draw[->] (-.07,1.6) -- (-.07,.4);
	\draw[->] (.07,1.6) -- (.07,.4);
	
	\draw[->] (.4,2.07) -- (1.7,2.07);
	\draw[->] (.4,2-.07) -- (1.7,2-.07);
	
	\draw[->] (2-.07,1.6) -- (2-.07,.4);
	\draw[->] (2+.07,1.6) -- (2+.07,.4);
\end{tikzpicture}
\end{center}
Here, the left and right sides are the given groupoid structures, while the top and bottom are trivial groupoids.
\end{ex}

\begin{ex}\label{ex:dinertia}
Again for any Lie groupoid $G \rightrightarrows M$, there is a double Lie groupoid structure on
\begin{center}
\begin{tikzpicture}[>=angle 90]
	\node (LL) at (0,0) {$M \times M$};
	\node (LR) at (3,0) {$M.$};
	\node (UL) at (0,2) {$G \times G$};
	\node (UR) at (3,2) {$G$};

	\tikzset{font=\scriptsize};
	\draw[->] (.9,.07) -- (2.6,.07);
	\draw[->] (.9,-.07) -- (2.6,-.07);
	
	\draw[->] (-.07,1.6) -- (-.07,.4);
	\draw[->] (.07,1.6) -- (.07,.4);
	
	\draw[->] (.9,2.07) -- (2.5,2.07);
	\draw[->] (.9,2-.07) -- (2.5,2-.07);
	
	\draw[->] (3-.07,1.6) -- (3-.07,.4);
	\draw[->] (3.07,1.6) -- (3.07,.4);
\end{tikzpicture}
\end{center}
Here, the right side is the given groupoid structure, the top and bottom are pair groupoids, and the left is a product groupoid.
\end{ex}

For what follows we will need a consistent labeling of the structure maps involved in the various groupoids considered. First, the source, target, unit, inverse, and product of the right groupoid $V \rightrightarrows M$ are respectively
\[
r_V(\cdot), \ell_V(\cdot), 1^V_\cdot, i_V(\cdot), m_V(\cdot,\cdot) \text{ or } \cdot \circ_V \cdot
\]
The structure maps of $H \rightrightarrows M$ will use the same symbols with $V$ replaced by $H$. Now, the structure maps of the top and left groupoids will use the same symbol as those of the \emph{opposite} structure with a tilde on top; so for example, the structure maps of $D \rightrightarrows H$ are
\[
\widetilde{r}_V(\cdot), \widetilde{\ell}_V(\cdot), \widetilde{1}^V_\cdot, \widetilde{i}_V(\cdot), \wt m_V(\cdot,\cdot) \text{ or } \cdot \widetilde{\circ}_V \cdot
\]

To emphasize: the structure maps of the groupoid structure on $D$ \emph{over} $V$ use an $H$, and those of the groupoid structure on $D$ \emph{over} $H$ use a $V$. This has a nice practical benefit in that it is simpler to keep track of the various relations these maps satisfy; for example, the maps $\widetilde{r}_{H}$ and $r_{H}$ give the groupoid homomorphism
\begin{center}
\begin{tikzpicture}[>=angle 90,scale=.75]
	\node (LL) at (0,0) {$H$};
	\node (LR) at (2,0) {$M,$};
	\node (UL) at (0,2) {$D$};
	\node (UR) at (2,2) {$V$};

	\tikzset{font=\scriptsize};
	\draw[->] (UL) to node [above] {$\widetilde{r}_{H}$} (UR);
	
	\draw[->] (-.07,1.6) -- (-.07,.4);
	\draw[->] (.07,1.6) -- (.07,.4);
	
	\draw[->] (2-.07,1.6) -- (2-.07,.4);
	\draw[->] (2+.07,1.6) -- (2+.07,.4);
	
	\draw[->] (LL) -- node [above] {$r_{H}$} (LR);
\end{tikzpicture}
\end{center}
so for instance we have: $\ell_V(\wt r_H(s)) = r_H(\wt \ell_V(s))$, $i_V(\wt r_H(s)) = \wt r_H(\wt i_V(s))$, $\wt r_H(\wt 1^V_v) = 1^V_{r_H(v)}$, etc.

We denote elements of $D$ as squares with sides labeled by the possible sources and targets:
\begin{center}
\begin{tikzpicture}[scale=.75]
	\draw (0,0) rectangle (2,2);
	\node at (1,1) {$s$};

	\tikzset{font=\scriptsize};
	\node at (-.6,1) {$\wt\ell_V(s)$};
	\node at (1,2.4) {$\wt r_{H}(s)$};
	\node at (2.6,1) {$\wt r_V(s)$};
	\node at (1,-.4) {$\wt\ell_{H}(s)$};
\end{tikzpicture}
\end{center}
so that the left and right sides are the target and source of the left groupoid structure while the top and bottom sides are the source and target of the top groupoid structure. With this notation, composition in the left groupoid $D \rightrightarrows H$ can be viewed as ``horizontal concatenation'' and while composition in the top groupoid $D \rightrightarrows V$ is ``vertical concatenation''. The compatibility between the two groupoid products on $D$ can then be expressed as saying that composing vertically and then horizontally in
\begin{center}
\begin{tikzpicture}
	\draw (0,0) rectangle (1,1);
	\node at (.5,.5) {$a$};
	\draw (1,0) rectangle (2,1);
	\node at (1.5,.5) {$b$};
	\draw (0,1) rectangle (1,2);
	\node at (.5,1.5) {$c$};
	\draw (1,1) rectangle (2,2);
	\node at (1.5,1.5) {$d$};
\end{tikzpicture}
\end{center}
produces the same result as composing horizontally and then vertically, whenever all compositions involved are defined.

\begin{defn}
The \emph{core} of a double Lie groupoid $D$ is the submanifold $C$ of elements of $D$ whose sources are both units; that is, the set of elements of the form
\begin{center}
\begin{tikzpicture}[scale=.75]
	\node at (1,1) {$s$};
	
	\tikzset{font=\scriptsize};
	\draw (0,0) rectangle (2,2);
	\node at (-.7,1) {$\wt\ell_{V}(s)$};
	\node at (1,2.4) {$1^V_m$};
	\node at (2.6,1) {$1^{H}_m$};
	\node at (1,-.4) {$\wt\ell_H(s)$};
\end{tikzpicture}
\end{center}
The condition on the double source map in the definition of a double Lie groupoid ensures that $C$ is a submanifold of $D$.
\end{defn}

\begin{thrm}[Brown-Mackenzie, \cite{BM}]\label{core-grpd}
The core of a double Lie groupoid $D$ has a natural Lie groupoid structure over the double base $M$.
\end{thrm}

The groupoid structure on the core comes from a combination of the two groupoid structures on $D$. Explicitly, the groupoid product of two elements $s$ and $s'$ in the core can be expressed as composing vertically and then horizontally (or equivalently  horizontally and then vertically) in the following diagram:
\begin{center}
\begin{tikzpicture}
	\draw (0,0) rectangle (1.5,1);
	\node at (.75,.5) {$s$};
	\draw (1.5,0) rectangle (3,1);
	\node at (2.25,.5) {$\wt 1^{H}_{\wt\ell_{H}(s')}$};
	\draw (0,1) rectangle (1.5,2);
	\node at (.75,1.5) {$\wt 1^V_{\wt\ell_V(s')}$};
	\draw (1.5,1) rectangle (3,2);
	\node at (2.25,1.5) {$s'$};
\end{tikzpicture}
\end{center}

The core groupoid of Example~\ref{ex:dmain} is the trivial groupoid $M \rightrightarrows M$ while that of Example~\ref{ex:dinertia} is $G \rightrightarrows M$ itself.

\begin{defn}
A \emph{symplectic double groupoid} is a double Lie groupoid $D$ where $D$ is equipped with a symplectic structure making the top and left groupoid structures in~(\ref{defn:dbl-grpd}) symplectic groupoids.
\end{defn}

The symplectic structure on $D$ endows the core with additional structure as follows:

\begin{thrm}[Mackenzie \cite{M}]
The core $C$ of a symplectic double groupoid $D$ is a symplectic submanifold of $D$ and the induced groupoid structure on $C \rightrightarrows M$ is that of a symplectic groupoid.
\end{thrm}

\begin{ex}\label{ex:main}
For any groupoid $G \rightrightarrows M$, there is a symplectic double groupoid structure on $T^*G$ of the form
\begin{center}
\begin{tikzpicture}[>=angle 90]
	\node (LL) at (0,0) {$A^*$};
	\node (LR) at (2,0) {$M$};
	\node (UL) at (0,2) {$T^*G$};
	\node (UR) at (2,2) {$G$};

	\tikzset{font=\scriptsize};
	\draw[->] (.4,.07) -- (1.6,.07);
	\draw[->] (.4,-.07) -- (1.6,-.07);
	
	\draw[->] (-.07,1.6) -- (-.07,.4);
	\draw[->] (.07,1.6) -- (.07,.4);
	
	\draw[->] (.6,2.07) -- (1.7,2.07);
	\draw[->] (.6,2-.07) -- (1.7,2-.07);
	
	\draw[->] (2-.07,1.6) -- (2-.07,.4);
	\draw[->] (2+.07,1.6) -- (2+.07,.4);
\end{tikzpicture}
\end{center}
Here, $A$ is the Lie algebroid of $G \rightrightarrows M$, the right groupoid structure is the given one on $G$, the top and bottom are the natural groupoid structures on vector bundles given by fiber-wise addition, and the left structure is the induced symplectic groupoid structure on the cotangent bundle of a Lie groupoid. The core of this symplectic double groupoid is symplectomorphic to $T^*M$, and the core groupoid is simply $T^*M \rightrightarrows M$.
\end{ex}

\begin{ex}\label{ex:inertia}
Again for any groupoid $G \rightrightarrows M$, there is a symplectic double groupoid structure on $\overline{T^*G} \times T^*G$ of the form
\begin{center}
\begin{tikzpicture}[>=angle 90]
	\node (LL) at (0,0) {$\overline{A^*} \times A^*$};
	\node (LR) at (3,0) {$A^*$};
	\node (UL) at (0,2) {$\overline{T^*G} \times T^*G$};
	\node (UR) at (3,2) {$T^*G$};

	\tikzset{font=\scriptsize};
	\draw[->] (.9,.07) -- (2.6,.07);
	\draw[->] (.9,-.07) -- (2.6,-.07);
	
	\draw[->] (-.07,1.6) -- (-.07,.4);
	\draw[->] (.07,1.6) -- (.07,.4);
	
	\draw[->] (1.2,2.07) -- (2.5,2.07);
	\draw[->] (1.2,2-.07) -- (2.5,2-.07);
	
	\draw[->] (3-.07,1.6) -- (3-.07,.4);
	\draw[->] (3.07,1.6) -- (3.07,.4);
\end{tikzpicture}
\end{center}
Here, the right side is the induced symplectic groupoid structure on $T^*G$, the top and bottom are pair groupoids, and the left is a product groupoid. The core is symplectomorphic to $T^*G$, and the core groupoid is $T^*G \rightrightarrows A^*$.
\end{ex}

Both of the above examples are special cases of the following result due to Mackenzie:

\begin{thrm}[Mackenzie \cite{M}]\label{thrm:ctdbl}
Let $D$ be a double Lie groupoid. Then the cotangent bundle $T^*D$ has a natural symplectic double groupoid structure
\begin{center}
\begin{tikzpicture}[>=angle 90]
	\node (LL) at (0,0) {$A_H^*D$};
	\node (LR) at (3,0) {$A^*C$};
	\node (UL) at (0,2) {$T^*D$};
	\node (UR) at (3,2) {$A_V^*D$};

	\tikzset{font=\scriptsize};
	\draw[->] (.5,.07) -- (2.5,.07);
	\draw[->] (.5,-.07) -- (2.5,-.07);
	
	\draw[->] (-.07,1.6) -- (-.07,.4);
	\draw[->] (.07,1.6) -- (.07,.4);
	
	\draw[->] (.6,2.07) -- (2.5,2.07);
	\draw[->] (.6,2-.07) -- (2.5,2-.07);
	
	\draw[->] (3-.07,1.6) -- (3-.07,.4);
	\draw[->] (3+.07,1.6) -- (3+.07,.4);
\end{tikzpicture}
\end{center}
where $A_V^*D$ and $A_H^*D$ are the duals of the Lie algebroids of $D \rightrightarrows V$ and $D \rightrightarrows H$ respectively, and $A^*C$ is the dual of the Lie algebroid of the core groupoid $C \rightrightarrows M$. The core of this symplectic double groupoid is symplectomorphic to $T^*C$.
\end{thrm}

Example \ref{ex:main} arises from applying this theorem to the double groupoid of Example~\ref{ex:dmain}
and Example \ref{ex:inertia} arises from the double groupoid of Example~\ref{ex:dinertia}.

\section{Realizing the Core via Reduction}\label{sec:core-red}
We now describe the procedure for producing the core of a symplectic double groupoid, and indeed the symplectic groupoid structure on the core, via symplectic reduction. 

Let $D$ be a symplectic double groupoid. The unit submanifold $1^V M$ of the groupoid structure on $V$ is coisotropic in $V$ for the Poisson structure induced by the symplectic groupoid $D \rightrightarrows V$. Hence its preimage
\[
X := \widetilde{r}_{H}^{-1}(1^{V} M) \subseteq D
\]
under the source of the top groupoid (which is a Poisson map) is a coisotropic submanifold of $D$. In square notation $X$ consists of those elements of the form:
\begin{center}
\begin{tikzpicture}[scale=.75]
	\node at (1,1) {$s$};

	\tikzset{font=\scriptsize};
	\draw (0,0) rectangle (2,2);
	\node at (-.6,1) {$\wt\ell_V(s)$};
	\node at (1,2.4) {$1^V_m$};
	\node at (2.6,1) {$\wt r_V(s)$};
	\node at (1,-.4) {$\wt\ell_{H}(s)$};
\end{tikzpicture}
\end{center}
Note that the core $C$ of $D$ sits inside of $X$. Similarly, by doing the same with the transpose of $D$ we produce the coisotropic submanifold
\[
Y := \wt r_V^{-1}(1^{H}M) \subseteq D
\]
of $D$, which also contains the core.

\begin{ex}
Consider the symplectic double groupoid of Example \autoref{ex:main}. The submanifold $X$ in this case is the restricted cotangent bundle $T^*G|_M \subseteq T^*G$. As described in Example~\ref{ex:std-reds}, the reduction of this coisotropic by its characteristic foliation is the core $T^*M$, and the resulting reduction relation
\[
T^*G \red T^*M
\] 
is the transpose of the cotangent lift $T^*e$ of the unit embedding $e: M \to G$ of the groupoid $G \rightrightarrows M$.

Now, performing this procedure using the left groupoid is more interesting. Recall that the source map $\wt r$ of the groupoid $T^*G \rightrightarrows A^*$ is determined by the requirement that
\[
\wt r(g,\xi)\big|_{\ker d\ell_{e(r(g))}} = (dL_g)_{e(r(g))}^*\left(\xi\big|_{\ker d\ell_g}\right)
\]
where $L_g: \ell^{-1}(r(g)) \to \ell^{-1}(\ell(g))$ is left groupoid multiplication by $g$ and we identify $A^* \cong N^*M$. Here $\ell$ and $r$ are the target and source maps of $G \rightrightarrows M$ respectively. Thus, we see that $(g,\xi)$ maps to a unit $(r(g),0) \in A^*$ of the bottom groupoid (note that the unit of the bottom groupoid is given by the zero section $0_M$ of $A^*$) if and only if $\xi|_{\ker d\ell_g} = 0$, and hence equivalently if and only if $\xi$ is in the image of $d\ell_g^*$. 

Therefore, the coisotropic submanifold $Y := \wt r^{-1}(0_M)$ of $T^*G$ is equal to $N^*\F_\ell$, where $\F_\ell$ is the foliation of $G$ given by the $\ell$-fibers of the groupoid $G \rightrightarrows M$. According to Example~\ref{ex:std-reds}, the reduction of this is also the core $T^*M$, and the reduction relation $T^*G \red T^*M$ turns out to be the cotangent lift $T^*\ell$. 
\end{ex}

These results generalize in the following way to an arbitrary symplectic double groupoid. We first have the following explicit description of the characteristic foliation $X^\perp$ of $X := \widetilde{r}_{H}^{-1}(1^V M)$:

\begin{lem}\label{lem:fol}
The leaf of the foliation $X^\perp$ containing $s \in X$ is given by
\begin{equation}\label{char-fol}
X^\perp_s := \{ s\ \wt\circ_{H}\,\wt 1^V_\lambda\ |\ \lambda \in \ell_{H}^{-1}(m)\},
\end{equation}
where $m = r_V(\wt r_{H}(s))$.
\end{lem}

\begin{proof}
To show that the characteristic foliation is as claimed, we must show that
\[
T_s (X^\perp_s) = (T_s X)^\perp
\]
where $(T_s X)^\perp$ is the symplectic orthogonal of $T_s X$ in $T_sD$. A dimension count shows that these two spaces have the same dimension, so we need only show that the former is contained in the latter. Thus we must show that if $Y \in T_s(X^\perp_s)$, then
\[
\omega_s(Y,V) = 0
\]
for any $Z \in T_s X$ where $\omega$ is the symplectic form on $D$. Since
\[
T_s X = (d\wt r_{H})_s^{-1}\left(T_{1^V_m}(1^V M)\right),
\]
this means that $Z$ satisfies $(d\wt r_{H})_s Z \in T_{1^V_m}(1^V M)$.

We use the following explicit description of $T_s(X^\perp_s)$. The elements of $X_s^\perp$ are parametrized by $\ell_{H}^{-1}(m)$, so we have that $X^\perp_s$ is the image of the map
\[
\ell_{H}^{-1}(m) \to D
\]
given by the composition
\[
\lambda \mapsto \wt 1^V_\lambda \mapsto \wt L_s^{H} (\wt 1^V_\lambda)
\]
where $\wt L_s^{H}$ is left-multiplication by $s$ in the top groupoid. Taking differentials at $\lambda = 1^{H}_m$ then gives the explicit description of $T_s(X^\perp_s)$ we want; in particular, we can write $Y \in T_s(X^\perp_s)$ as
\begin{displaymath}
Y = (d \wt L_s^{H})_{\wt 1^V(1^{H}_m)}(d \wt 1^V)_{1^{H}_m} W
\end{displaymath}
for some $W \in \ker (d\ell_{H})_{1^{H}_m}$.

First we consider the case where $s$ is a unit of the top groupoid structure, so suppose that $s = \wt 1^{H}(1^V_m)$. In this case $\wt L_s^H$ is simply the identity, so the expression for $Y$ above becomes
\[
Y = (d \wt 1^V)_{1^{H}_m} W.
\]
Using the splitting
\[
TD|_V = TV \oplus \ker d\wt r_{H}|_V
\]
we can write $Z \in T_sX$ as
\[
Z = (d\wt 1^{H})_{1^V_m}(d\wt r_{H})_s Z + [Z- (d\wt 1^{H})_{1^V_m}(d\wt r_{H})_s Z],
\]
where the first term is tangent to the units of the top groupoid and the second term is in $\ker (d\wt r_{H})_s$, and so tangent to the $\wt r_H$-fiber through $s$. Then we have
\[
\omega_s (Y,Z) = \omega_s ((d \wt 1^V)_{1^{H}_m} W,(d\wt 1^{H})_{1^V_m}(d\wt r_{H})_s Z) +  \omega_s ((d \wt 1^V)_{1^{H}_m} W,[V - (d\wt 1^{H})_{1^V_m}(d\wt r_{H})_s Z]).
\]
Since $W \in \ker(d\ell_{H})_{1^{H}_m}$, one can check that $Y = (d \wt 1^V)_{1^{H}_m} W$ is tangent to the $\wt \ell_H$-fiber through $s$, so the second term above vanishes since source and target fibers of a symplectic groupoid are symplectically orthogonal to each other. By the defining property of $Z$, we have $(d\wt r_{H})_s Z = (d 1^V)_m z$ for some $z \in T_m M$. Using this and the fact that $\wt 1^{H} \circ 1^V = \wt 1^V \circ 1^{H}$ (which follows from the double groupoid compatibilities), we can write the first term above as
\[
\omega_s ((d \wt 1^V)_{1^{H}_m} W,(d \wt 1^V)_{1^{H}_m}(d 1^{H})_m z),
\]
which vanishes since the embedding $\wt 1^V: H \to D$ is lagrangian. Thus $\omega_s(Y,Z) = 0$ as was to be shown. 

Now, for the general case, given $s \in X$ set $p := \wt 1_H(1^V_m)$ and choose a local lagrangian bisection $B$ of the top groupoid $D \rightrightarrows V$ containing $s$. Recall that this is a lagrangian submanifold of $D$ such that the restrictions
\[
\wt\ell_H|_B: B \to U \text{ and } \wt r_H|_B: B \to U'
\]
are diffeomorphisms of $B$ onto open subsets $U$ and $U'$ of $V$. This then defines a left-multiplication $L_B: \wt\ell_H^{-1}(U') \to \wt\ell_H^{-1}(U)$ between the open submanifolds $\wt\ell_H^{-1}(U')$ and $\wt\ell_H^{-1}(U)$ of $D$ by
\[
L_B(d) = (\wt r_H|_B)^{-1}(\wt\ell_H(d))\, \wt\circ_H\, d,
\]
which sends the leaf $X^\perp_p$ of~(\ref{char-fol}) to $X^\perp_s$. It is known that $L_B$ is actually a symplectomorphism~\cite{X} and so sends the symplectic orthogonal $(T_p X)^\perp$ at $p$ to $(T_s X)^\perp$. Since $(T_{p} X)^\perp = T_{p}(X_{p}^\perp)$ by what we showed previously and $T_{p}(X_{p}^\perp)$ is sent to $T_s(X^\perp_s)$, we have our result.
\end{proof}

\begin{rmk}
To emphasize, it is the entire double groupoid structure that makes this explicit description possible. As a contrast, given a symplectic realization $f: S \to P$ of a Poisson manifold $P$ and a coisotropic submanifold $N \subseteq P$, the characteristic foliation on $f^{-1}(N) \subseteq S$ is not easily described.
\end{rmk}

In square notation, the leaf of the characteristic foliation through $s$ consists of elements of the form
\begin{center}
\begin{tikzpicture}[scale=.85]
	\node at (1,1) {$s\,\,\wt\circ_{H}\wt 1^V_\lambda$};

	\tikzset{font=\scriptsize};
	\draw (0,0) rectangle (2,2);
	\node at (-1.2,1) {$\wt\ell_V(s) \circ_{H} \lambda$};
	\node at (1,2.4) {$1^V_{r_{H}(\lambda)}$};
	\node at (3.2,1) {$\wt r_V(s) \circ_{H} \lambda$};
	\node at (1,-.4) {$\wt\ell_{H}(s)$};
\end{tikzpicture}
\end{center}
where $\lambda \in \ell_{H}^{-1}(m)$. Switching the roles of the top and left groupoids, we get that the leaves of the characteristic foliation of $Y := \wt r_V^{-1}(1^{H}M)$ are given by
\[
Y^\perp_s := \{s\ \wt\circ_V\,\wt 1^{H}_{\lambda}\ |\ \lambda \in \ell_V^{-1}(m)\},
\]
where $m = r_{H}(\wt r_V(s))$

Returning to the characteristic leaves of $X$, note that such a leaf intersects the core in exactly one point, since there is only one choice of $\lambda$ which will make $\wt r_V(s) \circ_{H} \lambda$ a unit for the left groupoid structure, namely $\lambda = i_{H}(\wt r_V(s))$. Thus, the core forms a cross section to the characteristic foliation of $X$ and we conclude that the leaf space $X/X^\perp$ of this characteristic foliation can be identified with the core. This leaf space is then naturally symplectic, and we have:

\begin{prop}
The symplectic structure on the core obtained via the reduction above agrees with the one $C$ inherits as a symplectic submanifold of $D$.
\end{prop}

\begin{proof}
Let $i: C \to D$ be the inclusion of the core into $D$, and let $\pi: X \to C$ be the surjective submersion sending $s \in X$ to the characteristic leaf containing it, where we have identified the leaf space $X/X^\perp$ with $C$ in the above manner. Let $s: C \to X$ be the inclusion of the core into $X$; this is a section of $\pi$. Finally, let $j: X \to D$ be the inclusion of $X$ into $D$.

The symplectic form $\omega_C$ on $C$ obtained via reduction is characterized by the property that $j^*\omega = \pi^*\omega_C$. Since $i = j \circ s$, we have
\begin{align*}
i^*\omega &= (j \circ s)^*\omega \\
&= s^*(j^*\omega) \\
&= s^*(\pi^*\omega_C) \\
&= (\pi \circ s)^*\omega_C,
\end{align*}
which equals $\omega_C$ since $s$ is a section of $\pi$. This proves the claim.
\end{proof}

Explicitly, the reduction relation $\Lambda: D \red C$ is given by
\begin{equation}\label{red}
\Lambda: s \mapsto s\, \wt\circ_{H} \wt 1^V_{i_{H}(\wt r_V(s))} \text{ for } s \in \wt r_{H}^{-1}(1^V M).
\end{equation}
For future reference, the transpose relation $\Lambda^t: C \cored D$ (which is a coreduction) is given by
\begin{equation}\label{redT}
\Lambda^t: s \mapsto s \wt\circ_{H} \wt 1^V_\lambda, \text{ for $\lambda \in H$ such that } \ell_{H}(\lambda) = r_V(\wt r_{H}(s)).
\end{equation}

The same then holds if we consider the tranposed double groupoid structure, so that the reduction of the coisotropic submanifold $Y = \wt r_V^{-1}(1^{H}M)$ of $D$ can also be identified with the core via the same maps, simply exchanging the roles of $V$ and $H$. The reduction relation $D \red C$ obtained by reducing $Y$ is explicitly given by
\begin{displaymath}
s \mapsto s\,\wt\circ_V\,\wt 1^{H}_{i_{V}(\wt r_{H}(s))} \text{ for } s \in \wt r_V^{-1}(1^{H} M).
\end{displaymath}

Similar results hold for the preimages of units under the target maps. To be clear, let $Z$ now be $\wt\ell_{H}^{-1}(1^V M)$, the preimage of the units of $V$ under the target map of the top groupoid. This is again coisotropic in $D$, and the leaf of the characteristic foliation through a point $s \in Z$ is now given by
\[
\{\wt 1^V_{\lambda}\,\wt\circ_{H}\,s\ |\ \lambda \in r_{H}^{-1}(m)\}
\]
where $m = \ell_V(\wt\ell_{H}(s))$. Similar to the above, we can now easily identify the reduction of $Z$ with the set of elements of $D$ of the form
\begin{center}
\begin{tikzpicture}[scale=.75]
	\node at (1,1) {$s$};

	\tikzset{font=\scriptsize};
	\draw (0,0) rectangle (2,2);
	\node at (-.6,1) {$1^{H}_m$};
	\node at (1,2.4) {$\wt r_{H}(s)$};
	\node at (2.7,1) {$\wt r_V(s)$};
	\node at (1,-.4) {$1^V_m$};
\end{tikzpicture}
\end{center}
which we might call the ``left-core'' $C_L$ of $D$ to distinguish it from the ``right-core'' $C_R\ (:= C)$ previously defined. However, the left-core $C_L$ can be identified with $C_R$ using the composition of the two groupoid inverses on $D$:
\[
\wt i_V \circ \wt i_{H}: C_L \to C_R,
\]
so we again get that the reduction of $Z$ is symplectomorphic to the core $C$. Considering the transpose of $D$, we find that the same is true for the preimage of the units of $H$ under the target of the left groupoid.

In summary, we have shown:
\begin{thrm}
Let $D$ be a symplectic double groupoid with core $C$. Then the reductions of the coisotropic submanifolds
\[
\wt r_{H}^{-1}(1^V M),\ \wt r_V^{-1}(1^{H}M),\ \wt \ell_{H}^{-1}(1^V M), \text{ and } \wt \ell_V^{-1}(1^{H}M)
\]
of $D$ are symplectomorphic to $C$.
\end{thrm}

\begin{ex}
Let us return to Example \ref{ex:main}. The target of the top groupoid is the same as the source, so the above reduction procedure produces the same reduction relation
\[
(T^*e)^t: T^*G \red T^*M
\]
as before. A similar computation to that carried out for the source of the left groupoid $T^*G \rightrightarrows A^*$ shows that the coisotropic submanifold $\wt \ell^{-1}(0_M)$ of $T^*G$ is $N^*\F_r$, where $\F_r$ is the foliation of $G$ given by the $r$-fibers of $G \rightrightarrows M$, and that the reduction relation
\[
T^*G \red T^*M
\]
obtained by reducing $N^*\F_r$ is then the cotangent lift $T^*r$.
\end{ex}

\begin{rmk}\label{gen-cot}
These examples show that in the case of the standard double groupoid structure on the cotangent bundle $T^*G$ of a Lie groupoid $G \rightrightarrows M$, the reduction relations arising from the various ways of realizing the core $T^*M$ are precisely the cotangent lifts of the structure maps of $G$. The reduction relations obtained by applying this procedure to a general cotangent double groupoid as in Theorem~\ref{thrm:ctdbl} can also be described as certain cotangent lifts; this will be elaborated on in~\cite{C}.
\end{rmk}

\begin{ex}
Consider the double groupoid of Example~\ref{ex:inertia}. The coisotropic submanifold $X$ of $\bar{T^*G} \times T^*G$ is $\bar{T^*G} \times A^*$, whose reduction is $T^*G$ since $A^*$ is lagrangian in $T^*G$. The reduction relation $\bar{T^*G} \times T^*G \red T^*G$ is $id \times A^*$.

In the transposed double groupoid, the coisotropic submanifold $Y$ of $\bar{T^*G} \times T^*G$ is the fiber product of $\wt r: T^*G \to A^*$ with itself. Let $((g,\xi),(h,\eta))$ be an element of this fiber product. Then in particular $r(g) = r(h)$, so $gh^{-1}$ is defined under the multiplication on $G$. The leaf of the characteristic foliation of $Y$ containing $((g,\xi),(h,\eta))$ consists of elements of the form
$$((g,\xi)\circ(k,\omega),(h,\eta)\circ(k,\omega))$$
where $\circ$ is the product on $T^*G$ and $(k,\omega) \in T^*G$ satisfies $\wt\ell(k,\omega) = \wt r(g,\xi) = \wt r(h,\eta)$. The element of the core $T^*G$ associated with this leaf is
$$(g,\xi) \circ (h^{-1},di^*_{h^{-1}}\eta)$$
where $i$ is the inverse of $G$. As mentioned at the end of Remark~\ref{gen-cot}, the resulting reduction relation $\bar{T^*G} \times T^*G \red T^*G$ can be described as a cotangent lift---in particular, it is the cotangent lift of the \emph{smooth relation} $G \times G \to G$ given by $(g,h) \mapsto gh^{-1} \text{ for $g,h \in G$ such that } r(g)=r(h)$.
\end{ex}

\begin{ex}
In a symplectic double groupoid of the form:
\begin{center}
\begin{tikzpicture}[>=angle 90,scale=.85,baseline=(current  bounding  box.center)]
	\node at (0,0) {$P^*$};
	\node at (2,0) {$pt,$};
	\node at (0,2) {$S$};
	\node at (2,2) {$P$};

	\tikzset{font=\scriptsize};
	\draw[->] (.4,.08) -- (1.6,.08);
	\draw[->] (.4,-.08) -- (1.6,-.08);
	
	\draw[->] (-.08,1.6) -- (-.08,.4);
	\draw[->] (.08,1.6) -- (.08,.4);
	
	\draw[->] (.4,2+.08) -- (1.6,2+.08);
	\draw[->] (.4,2-.08) -- (1.6,2-.08);
	
	\draw[->] (2-.08,1.6) -- (2-.08,.4);
	\draw[->] (2+.08,1.6) -- (2+.08,.4);
\end{tikzpicture}
\end{center}
$P$ and $P^*$ are dual Poisson Lie groups. The coisotropic submanifold $X := \widetilde{r}_{P^*}^{-1}(1^{P} pt)$ of $S$ in this case is lagrangian and can be identified with the embedding of $P^*$ into $S$. Being lagrangian, its reduction is a point which is indeed the core of $S$. The reduction relation $S \red pt$ is given by $P^*$. The transposed structure produces the canonical relation $S \red pt$ given by $P$ as a lagrangian submanifold of $S$.
\end{ex}

We thus have multiple ways of recovering the core of $D$ by reducing certain coisotropic submanifolds, and each such way produces a canonical relation $D \red C$. We now show that the left groupoid structure $D \rightrightarrows H$ descends to a well-defined groupoid structure on $C$ which agrees with Brown and Mackenzie's structure from Theorem~\ref{core-grpd}.

First, the groupoid product on $C \rightrightarrows M$ should come from a relation $C \times C \to C$, and using the reduction $\Lambda: D \red C$ we have defined there is a natural way of obtaining this desired relation:

\begin{prop}
The canonical relation $m: C \times C \to C$ given by the composition
\begin{equation}\label{core-product}
\begin{tikzpicture}[>=angle 90,baseline=(current  bounding  box.center)]
	\node (UL) at (0,0) {$C \times C$};
	\node (UR) at (3,0) {$D \times D$};
	\node (LL) at (6,0) {$D$};
	\node (LR) at (8,0) {$C,$};

	\tikzset{font=\scriptsize};
	\draw[->] (UL) to node [above] {$\Lambda^t \times \Lambda^t$} (UR);
	\draw[->] (UR) to node [above] {$\wt m_V$} (LL);
	\draw[->] (LL) to node [above] {$\Lambda$} (LR);
\end{tikzpicture}
\end{equation}
where $\wt m_V$ is the product of the left grouopid $D \rightrightarrows H$ viewed as a canonical relatoin, is Brown and Mackenzie's groupoid product on $C$.
\end{prop}

\begin{proof}
Let $s, s' \in C$. Applying the relation $\Lambda^t \times \Lambda^t$ gives
\[
(s,s') \mapsto (s\,\wt\circ_{H}\,\wt 1^V_\lambda, s'\,\wt\circ_{H}\,\wt 1^V_{\lambda'})
\]
where $\lambda, \lambda'$ are as in~(\ref{redT}). Now, these are composable under $\wt m_V$ when
\[
\wt r_V(s\,\wt\circ_{H}\,\wt 1^V_\lambda) = \wt r_V(s) \circ_{H} \lambda = \lambda \text{ equals } \wt \ell_V(s'\,\wt\circ_{H}\,\wt 1^V_{\lambda'}) = \wt\ell_V(s') \circ_{H} \lambda',
\]
where we have used the fact that $\wt r_V(s)$ is a unit. From this we get the condition that
\[
\lambda = \wt\ell_V(s') \circ_{H}\lambda'.
\]
The relation $\wt m_V$ then produces
\[
(s\,\wt\circ_{H}\,\wt 1^V_\lambda) \wt\circ_V (s'\,\wt\circ_{H}\,\wt 1^V_{\lambda'}) = \left(s\,\wt\circ_{H}\,\wt 1^V_{\wt\ell_V(s') \circ_{H}\lambda'}\right) \wt\circ_V (s'\,\wt\circ_{H}\,\wt 1^V_{\lambda'}).
\]

Now, applying the map $\wt r_H$ to this gives
$$\wt r_H(\wt 1^V_{\wt\ell_V(s') \circ_{H}\lambda'})\,\wt \circ_V\,\wt r_H(\wt 1^V_{\lambda'}) = 1^V_{r_H(\lambda')}\,\circ_V\,1^V_{r_H(\lambda')} = 1^V_{r_H(\lambda')},$$ so that the element above is already in the domain $X$ of $\Lambda$. Applying the final relation $\Lambda$ then gives
$$\left[\left(s\,\wt\circ_{H}\,\wt 1^V_{\wt\ell_V(s') \circ_{H}\lambda'}\right) \wt\circ_V (s'\,\wt\circ_{H}\,\wt 1^V_{\lambda'})\right]\,\wt\circ_H\,\wt 1^V_{i_H(\lambda')}.$$
Expressing this in square notation, we have:
\begin{center}
\begin{tikzpicture}
	\draw (0,0) rectangle (2.75,1);
	\node at (1.375,.5) {$s\,\wt\circ_{H}\,\wt 1^V_{\wt\ell_V(s') \circ_{H}\lambda'}$};
	\draw (2.75,0) rectangle (4.5,1);
	\node at (3.625,.5) {$s'\,\wt\circ_{H}\,\wt 1^V_{\lambda'}$};
	\draw (0,1) rectangle (4.5,2);
	\node at (2.25,1.5) {$\wt 1^V_{i_H(\lambda')}$};
	
	\node at (5,1) {$=$};
	
	\draw (5.5,0) rectangle (8.25,1);
	\node at (6.875,.5) {$s\,\wt\circ_{H}\,\wt 1^V_{\wt\ell_V(s') \circ_{H}\lambda'}$};
	\draw (8.25,0) rectangle (10.25,1);
	\node at (9.25,.5) {$s'\,\wt\circ_{H}\,\wt 1^V_{\lambda'}$};
	\draw (5.5,1) rectangle (8.25,2);
	\node at (6.875,1.5) {$\wt 1^V_{i_H(\lambda')}$};
	\draw (8.25,1) rectangle (10.25,2);
	\node at (9.25,1.5) {$\wt 1^V_{i_H(\lambda')}$};
	
	\node at (10.75,1) {$=$};
	
	\draw (11.25,.5) rectangle (13.25,1.5);
	\node at (12.25,1) {$s\,\wt\circ_H\,\wt 1^V_{\wt\ell_V(s')}$};
	\draw (13.25,.5) rectangle (14.25,1.5);
	\node at (13.75,1) {$s'$};
\end{tikzpicture}
\end{center}
where in the first step we decompose the top box horizontally and in the second we compose vertically. Using $s' = \wt 1^H_{\ell_H(s')}\,\wt\circ_H\, s'$, we can write this final expression as
\begin{center}
\begin{tikzpicture}
	\draw (-4,.5) rectangle (-2,1.5);
	\node at (-3,1) {$s\,\wt\circ_H\,\wt 1^V_{\wt\ell_V(s')}$};
	\draw (-2,.5) rectangle (-1,1.5);
	\node at (-1.5,1) {$s'$};
	
	\node at (-.5,1) {$=$};
	
	\draw (0,0) rectangle (1.5,1);
	\node at (.75,.5) {$s$};
	\draw (1.5,0) rectangle (3,1);
	\node at (2.25,.5) {$\wt 1^{H}_{\wt\ell_{H}(s')}$};
	\draw (0,1) rectangle (1.5,2);
	\node at (.75,1.5) {$\wt 1^V_{\wt\ell_V(s')}$};
	\draw (1.5,1) rectangle (3,2);
	\node at (2.25,1.5) {$s'$};
\end{tikzpicture}
\end{center}
which gives Brown and Mackenzie's core groupoid product as claimed.
\end{proof}

\begin{rmk}
The proof of this proposition demonstrates the following. Take two elements $s$ and $s'$ of the leaf space $C = X/X^\perp$ which we want to compose and pick some $d' \in X$ in the leaf $s'$. First, there is precisely one element $d \in X$ in the leaf $s$ which is composable under $\wt m_V$ with $d'$. Composing these gives the element $d\,\wt\circ_V\,d'$ of $X$, and the leaf which this determines in $C$ does not depend on the choice of $d'$. Denoting this leaf by $s \circ s'$ defines the groupoid product on $C$.
\end{rmk}

\begin{rmk}
Instead of using the relation $\Lambda: D \red C$ obtained by reducing $X = \wt r_H^{-1}(1^V M)$, we could have looked at the transposed double groupoid and used the relation $D \to C$ obtained by reducing $Y = \wt r_V^{-1}(1^H M)$. The resulting expression for the product on $C$ agrees with the alternate expression for the core groupoid product found in~\cite{BM}.
\end{rmk}

Note that $m: C \times C \to C$ is a canonical relation simply because we have defined it as a composition of such. The unit $e: pt \to C$ is obtained as the composition
\begin{equation}\label{core-unit}
\begin{tikzpicture}[>=angle 90,baseline=(current  bounding  box.center)]
	\node (UL) at (0,1) {$pt$};
	\node (UM) at (2,1) {$D$};
	\node (UR) at (4,1) {$C,$};

	\tikzset{font=\scriptsize};
	\draw[->] (UL) to node [above] {$H$} (UM);
	\draw[->] (UM) to node [above] {$\Lambda$} (UR);
\end{tikzpicture}
\end{equation}
where $H$ denotes the image of the lagrangian embedding $\wt 1^V: H \to D$. A direct computation shows that $\Lambda \circ H = M$ viewed as a lagrangian submanifold of $C$, as we expect for the unit submanifold of $C$.

Now, the left unit property of $e$ follows from the commutativity of the diagram
\begin{center}
\begin{tikzpicture}[>=angle 90]
	\node (L1) at (0,0) {$C$};
	\node (L2) at (3,0) {$D$};
	\node (U1) at (0,2) {$C \times C$};
	\node (U2) at (3,2) {$D \times D$};
	\node (U3) at (6,1) {$D$};
	\node (U4) at (9,1) {$C$};

	\tikzset{font=\scriptsize};
	\draw[->] (U1) to node [above] {$\Lambda^t \times \Lambda^t$} (U2);
	\draw[->] (U2) to node [above] {$\wt m_V$} (U3);
	\draw[->] (U3) to node [above] {$\Lambda$} (U4);
	\draw[->] (L1) to node [above] {$\Lambda^t$} (L2);
	\draw[->] (L1) to node [left] {$e \times id$} (U1);
	\draw[->] (L2) to node [left] {$H \times id$} (U2);
	\draw[->] (L2) to node [below] {$id$} (U3);
	\draw[->] (U1) to [bend left=30] node [above] {$m$} (U4);
	\draw[->] (L1) to [bend right=20] node [below] {$id$} (U4);
\end{tikzpicture}
\end{center}
where a direct computation shows that
$$\Lambda^t \circ \Lambda \circ H = H,$$
as required for the commutativity of the square on the left. This equality in particular says that the leaf of the characteristic foliation of $X$ through an element of $H \subseteq X$ consists only of elements of $H$. A similar diagram gives the right unit property of $e$. 

Finally, the $*$-structure $s: \overline{C} \to C$ on $(C,m,e)$ is given by the composition
\begin{equation}\label{core-star}
\begin{tikzpicture}[>=angle 90,baseline=(current  bounding  box.center)]
	\node (1) at (0,1) {$\overline{C}$};
	\node (2) at (2,1) {$\overline{D}$};
	\node (3) at (4,1) {$D$};
	\node (4) at (6,1) {$C.$};

	\tikzset{font=\scriptsize};
	\draw[->] (1) to node [above] {$\Lambda^t$} (2);
	\draw[->] (2) to node [above] {$\wt i_V$} (3);
	\draw[->] (3) to node [above] {$\Lambda$} (4);
\end{tikzpicture}
\end{equation}
A direct computation shows that $s^2 = id$, and we now verify that:

\begin{prop}
The canonical relation $s: \bar C \to C$ above agrees with Brown and Mackenzie's core groupoid inverse.
\end{prop}

\begin{proof}
First, for $s \in C$ we have
$$\wt i_V \circ \Lambda^t: s \mapsto \wt i_V\left(s\,\wt\circ_H\,\wt 1^V_\lambda\right) = \wt i_V(s)\,\wt\circ_H\,\wt 1^V_\lambda$$
where $\lambda$ is as in~(\ref{redT}). The element above is in the domain $X$ of $\Lambda$, and applying $\Lambda$ gives:
$$\left(\wt i_V(s)\,\wt\circ_H\,\wt 1^V_\lambda\right)\,\wt\circ_H\,\wt1^V_{i_H(\wt\ell_V(s) \circ_H \lambda)} = \wt i_V(s)\,\wt\circ_H\,\left(\wt 1^V_\lambda\,\wt\circ_H\,\wt1^V_{i_H(\lambda)}\wt\circ_H\wt 1^V_{i_H(\wt\ell_V(s))}\right) = \wt i_V(s)\,\wt\circ_H\,\wt1^V_{i_H(\wt\ell_V(s))},$$
which is the expression for the core groupoid inverse of $s$.
\end{proof}

\begin{rmk}
The proof of this proposition shows that the map $\wt i_V$ preserves the foliation $X^\perp$, meaning that it sends leaves to leaves. This is why $\wt i_V$ descends to a map $\bar C \to C$ on the leaf space, which is the core groupoid inverse. Doing the same after replacing $\Lambda$ with the relation $D \to C$ obtained by reducing $Y = \wt r_V^{-1}(1^H M)$ gives the alternate expression for the core groupoid inverse found in~\cite{BM}.
\end{rmk}

Putting this all together gives:

\begin{thrm}
For a symplectic double groupoid $D$ with core $C$, the left groupoid structure on $D \rightrightarrows H$ descends to a groupoid structure on the leaf space $C = X/X^\perp$ of the characteristic foliation of $X = \wt r_H^{-1}(1^VM)$. In particular, the canonical relations $(m,e,s)$ defined respectively by~(\ref{core-product}),(\ref{core-unit}), and (\ref{core-star}) respectively give $C$ the structure of a strongly positive symplectic $*$-monoid which agrees the standard symplectic core groupoid structure.
\end{thrm}

\begin{proof}
We have seen that $(C,m,e)$ is a symplectic monoid reproducing Brown and Mackenzie's core groupoid product. Hence all that remains to be shown is that $s$ is a strongly postitive $*$-structure for this monoid.

Consider the diagram:
\begin{center}
\begin{tikzpicture}[>=angle 90,scale=.75]
	\node (ULC) at (0,3) {$\overline{C} \times \overline{C}$};
	\node (UC) at (3,3) {$\overline{C} \times \overline{C}$};
	\node (URC) at (7,3) {$C \times C$};
	\node (LLC) at (0,-3) {$\overline{C}$};
	\node (LRC) at (7,-3) {$C$};

	\node (UL) at (0,1) {$\overline{D} \times \overline{D}$};
	\node (U) at (3,1) {$\overline{D} \times \overline{D}$};
	\node (UR) at (7,1) {$D \times D$};
	\node (LL) at (0,-1) {$\overline{D}$};
	\node (LR) at (7,-1) {$D$};

	\tikzset{font=\scriptsize};
	\draw[->] (ULC) to node [above] {$\sigma$} (UC);
	\draw[->] (UC) to node [above] {$s \times s$} (URC);
	\draw[->] (ULC) to [bend right=80] node [left] {$m$} (LLC);
	\draw[->] (URC) to [bend left=80] node [right] {$m$} (LRC);
	\draw[->] (LLC) to node [above] {$s$} (LRC);

	\draw[->] (UL) to node [above] {$\sigma$} (U);
	\draw[->] (U) to node [above] {$\wt i_V \times \wt i_V$} (UR);
	\draw[->] (UL) to node [left] {$\wt m_V$} (LL);
	\draw[->] (UR) to node [right] {$\wt m_V$} (LR);
	\draw[->] (LL) to node [above] {$\wt i_V$} (LR);
	
	\draw[->] (ULC) to node [left] {$\Lambda^t \times \Lambda^t$} (UL);
	\draw[->] (LL) to node [left] {$\Lambda$} (LLC);
	\draw[->] (URC) to node [right] {$\Lambda^t \times \Lambda^t$} (UR);
	\draw[->] (LR) to node [right] {$\Lambda$} (LRC);
\end{tikzpicture}
\end{center}
where $\sigma$ is the symplectomorphism exchanging components. The middle square commutes since $D$ is a symplectic $*$-monoid with product $\wt m_V$ and $*$-structure $\wt i_V$. The top and bottom squares commute since
$$\Lambda \circ \wt i_V = s \circ \Lambda$$
as a direct computation shows. This implies the commutativity of the outermost square, which says that $s$ is a $*$-structure for $(C,m,e)$. Similarly, the commutativity of the middle square in the diagram
\begin{center}
\begin{tikzpicture}[>=angle 90,scale=.75]
	\node (ULC) at (0,3) {$C \times \overline{C}$};
	\node (URC) at (4,3) {$C \times C$};
	\node (LLC) at (0,-3) {$pt$};
	\node (LRC) at (4,-3) {$C$};    

	\node (UL) at (0,1) {$D \times \overline{D}$};
	\node (UR) at (4,1) {$D \times D$};
	\node (LL) at (0,-1) {$pt$};
	\node (LR) at (4,-1) {$D$};

	\tikzset{font=\scriptsize};
	\draw[->] (ULC) to node [above] {$id \times s$} (URC);
	\draw[->] (LLC) to [bend left=80] node [left] {$$} (ULC);
	\draw[->] (URC) to [bend left=80] node [right] {$m$} (LRC);
	\draw[->] (LLC) to node [above] {$e$} (LRC);
	
	\draw[->] (UL) to node [above] {$id \times \wt i_V$} (UR);
	\draw[->] (LL) to node [left] {$$} (UL);
	\draw[->] (UR) to node [right] {$\wt m_V$} (LR);
	\draw[->] (LL) to node [above] {$H$} (LR);
	
	\draw[->] (UL) to node [left] {$\Lambda \times \Lambda$} (ULC);
	\draw[->] (URC) to node [right] {$\Lambda^t \times \Lambda^t$} (UR);
	\draw[->] (LLC) to node [left] {$$} (LL);
	\draw[->] (LR) to node [right] {$\Lambda$} (LRC);
\end{tikzpicture}
\end{center}
where $pt \to D \times \bar D$ and $pt \to C \times \bar C$ are given by the diagonals, follows from the fact that $\wt i_V$ is a strongly positive $*$-structure. One can check that the outermost square commute as well, which says that $s$ is strongly positive as claimed.
\end{proof}

\begin{rmk}
It is interesting to note that in the final diagram above, the outermost square commutes even though the top and bottom squares do not.
\end{rmk}

In a forthcoming paper~\cite{C} (see also~\cite{C1}) we make use of this reduction approach to the core in order to give a characterization of symplectic double groupoids in terms of the symplectic category. The canonical relations arising from these reductions will form the structure morphisms of a groupoid-like object $D \rightrightarrows C$ in the symplectic category.

\end{document}